\documentclass[a4paper, 12pt, oneside, notitlepage]{amsart}
\usepackage[margin=3cm]{geometry}
\usepackage{amsmath,amssymb,amsthm,graphicx,mathrsfs,bbm,url,stmaryrd,amscd,cancel}
\usepackage{amsthm}
\usepackage{wrapfig}
\usepackage{enumitem}
\usepackage{mathtools}
\usepackage[utf8]{inputenc}
 \usepackage[T1]{fontenc}
\usepackage[usenames,dvipsnames]{color}
\usepackage[colorlinks=true,linkcolor=Red,citecolor=Green]{hyperref}
\usepackage[super]{nth}
\usepackage[open, openlevel=2, depth=3, atend]{bookmark}
\hypersetup{pdfstartview=XYZ}
\usepackage[font=footnotesize]{caption}
\usepackage{a4wide}
\usepackage{tikz}
\usepackage{tikz-cd}

\usepackage{epstopdf}
 
\usepackage{hyperref}

\theoremstyle{plain}
\newtheorem{theorem}{Theorem}[section]
\newtheorem*{theorem*}{Theorem}

\newtheorem{lemma}[theorem]{Lemma}
\newtheorem{proposition}[theorem]{Proposition}
\newtheorem{corollary}[theorem]{Corollary}

\newtheorem{conjecture}[theorem]{Conjecture}

\theoremstyle{definition}
\newtheorem{definition}[theorem]{Definition}

\theoremstyle{remark}
\newtheorem{remark}[theorem]{Remark}

\numberwithin{equation}{section}

\newcommand{\C}{\mathbb{C}}
\newcommand{\R}{\mathbb{R}}
\newcommand{\Q}{\mathbb{Q}}
\newcommand{\Z}{\mathbb{Z}}

\newcommand{\M}{\mathcal{M}}
\newcommand{\N}{\mathbb{N}}

\newcommand{\V}{\mathbb{V}}

\newcommand{\HH}{\mathbb{H}}
\newcommand{\Ss}{\mathbb{S}}
\newcommand{\eps}{\varepsilon}

\newcommand{\mc}{\mathcal}

\newcommand{\comp}{\mathrm{comp}}

\DeclareMathOperator{\Op}{Op}
\DeclareMathOperator{\WF}{WF}

\DeclareMathOperator{\ran}{ran}

\DeclareMathOperator{\supp}{supp}

\newcommand{\be}{\begin{equation}}
\newcommand{\ee}{\end{equation}}

\title
[Singular set on nonpositively curved surfaces]
{On the singular set of the geodesic flow on nonpositively curved surfaces}

\author[Lefeuvre]{Thibault Lefeuvre}
\address{Université Paris-Saclay, CNRS, Laboratoire de mathématiques d’Orsay, 91405, Orsay, France.}
\email{thibault.lefeuvre1@universite-paris-saclay.fr}

\begin{document}

\begin{abstract}
We prove that the singular locus of the geodesic flow on the unit tangent bundle of a nonpositively curved surface enjoys an additional regularity property: its characteristic function belongs to the geometric anisotropic Sobolev space $H^{5/8-}_xH^{-13/16}_\theta$. In particular, this implies that if the geodesic flow is not ergodic with respect to the Liouville measure, then the zero set of the curvature must be ``thick'' in the sense that it supports a nonzero, nonnegative function in $H^{5/8-}$. The proof relies on semiclassical analysis and a new radial estimate for parabolic flows.
\end{abstract}

\maketitle

\section{Introduction}

\subsection{The ergodicity conjecture in nonpositive curvature} Let $(M,g)$ be a closed connected oriented nonpositively curved Riemannian surface of genus $\geq 2$. Let $\pi : SM \to M$ be the unit tangent bundle, $(\varphi_t)_{t \in \R}$ the geodesic flow on $SM$. The Liouville probability measure $\mu$ on $SM$ is a smooth invariant measure under the geodesic flow; it is thus natural to study the ergodicity of the geodesic flow with respect to it.

\subsubsection{Background} When the curvature is negative, the geodesic flow is Anosov and ergodicity dates back to Hopf \cite{Hopf-36} in constant curvature (it was generalized later by Anosov \cite{Anosov-67} in variable curvature and higher dimensions). While the geodesic flow is \emph{not} ergodic with respect to the Liouville measure on the flat torus, it is expected to be in genus $\geq 2$. However, this is still a widely open conjecture \cite[Question 6.1.1]{Burns-Matveev-21}. The present article proves new results towards this conjecture. 

Define the following closed flow-invariant \emph{singular set} (or \emph{locus}) of $SM$:
\begin{equation}
\Lambda := \left\{v \in SM ~:~ K(\pi(\varphi_tv))=0, \forall t \in \R\right\} \subset SM,
\end{equation}
where $K$ is the sectional curvature. It is a standard result due to Pesin (see \cite[Theorem 10.11]{Barreira-Pesin-13} for instance) that the geodesic flow is ergodic with respect to $\mu$ on the complement of $\Lambda$. The ergodicity conjecture is implied by the following folklore conjecture:

\begin{conjecture}
\label{conjecture:folklore}
The singular set $\Lambda$ has zero Liouville measure.
\end{conjecture}

Little is known on the set $\Lambda$: it is closed, flow-invariant and its complement is an open dense set with positive Liouville measure (see \cite[Theorem 10.11]{Barreira-Pesin-13}). Wu \cite{Wu-15} proved ergodicity under the assumption that the number of connected components of $\{K < 0\}$ is finite. Historically, the ergodicity of the geodesic flow was erroneously asserted in the 1970s, as documented in \cite[Theorem 10.7]{Pesin-77} and subsequent works \cite{Pesin-77-2, Pesin-78, Pesin-79, Pesin-81}. This claim was later used in the 1980s by various authors \cite{Ballmann-Brin-82, Burns-83, Ballmann-Brin-Eberlein-85}, though the fundamental flaw lay in the (still unknown) fact that $\Lambda$ has measure zero. 

Regarding the measure of maximal entropy it is known to be unique and ergodic on nonpositively-curved surfaces \cite{Knieper-98} (more generally, on all compact rank one manifolds). More recently, this was shown for certain metrics without conjugate points \cite{Climenhaga-Knieper-War-21}.

\subsubsection{Main result} The following theorem is the main result of the present paper:

\begin{theorem}
\label{theorem}
Let $(M,g)$ be a closed connected oriented Riemannian surface of genus $\geq 2$ with nonpositive sectional curvature. Then the indicator function $\mathbf{1}_\Lambda$ of $\Lambda$ satisfies:
\[
\mathbf{1}_\Lambda \in H^{5/8-}_x H^{-13/16}_\theta(SM).
\]
\end{theorem}

The Sobolev spaces $H^s_xH^{t}_\theta(SM)$ are introduced in \S\ref{ssection:spaces}; for now, let us simply say that they correspond to functions which are $H^s$ in the base variable $x \in M$, and $H^t$ in the velocity variable $\theta$ of the circle fibration $SM \to M$. More generally, Theorem \ref{theorem} holds for any $L^2$ function which is invariant by the geodesic flow. We believe that it should be possible to prove that for any $s > 0$, there exists $t > 0$ such that $\mathbf{1}_\Lambda \in H^{s}_x H^{-t}_\theta(SM)$.

\subsubsection{Consequences} We now explain a straightforward consequence of Theorem \ref{theorem}. Let $u_0 := \pi_* \mathbf{1}_\Lambda \in L^\infty(M)$ be the pushforward of the characteristic function, that is,
\begin{equation}
\label{equation:intro2}
u_0 (x) = \dfrac{1}{2\pi} \int_{S_xM} \mathbf{1}_\Lambda(x,\theta) |d\theta| = \mu_{S_xM}(\Lambda \cap S_xM),
\end{equation}
where $\mu_{S_xM}$ is the normalized arclength measure on the circle. Note that $\pi_* \mathbf{1}_\Lambda$ is supported in $\{K=0\}$. By construction, $u_0 \in L^\infty(M)$.

\begin{corollary}
\label{corollary}
The function $u_0$ lies in $H^{5/8-}(M)$.
\end{corollary}

In fact, this regularity holds for \emph{all} Fourier modes of $\mathbf{1}_{\Lambda}$ when decomposing $\mathbf{1}_\Lambda$ into Fourier modes in the circle fibers of $\pi : SM \to M$. The function $u_0$ corresponds to the $0$-th Fourier mode. Observe that, by the Sobolev trace theorem, since $5/8 > 1/2$, the restriction of $u_0$ to smooth curves is now well-defined. More generally, for any $a \in C^\infty(SM)$, one can define $u_a := \pi_*(a \mathbf{1}_\Lambda)$ and the same proof shows that $u_a \in H^{5/8-}(M)$. 

Since $u_0$ (and, more generally, $u_a$ for $a \in C^\infty(SM)$) is supported in $\{K=0\}$, Corollary \ref{corollary} implies that, if $\Lambda$ has positive Liouville measure, then the zero set $\{K=0\}$ is ``thick'' in the sense that it supports a nonzero nonnegative function of regularity $H^{5/8-}$. More precisely, the following holds:

\begin{corollary}
\label{corollary:ai}
Assume that $\Lambda \subset SM$ has positive Liouville measure. Let $I, J \subset \R$ be two closed intervals and $\Phi : I \times J \to U \subset M$ be a smooth coordinate chart such that
\[
\int_U u_a(x) |dx| > 0,
\]
where $|dx|$ is the Riemannian volume on $M$. Then there exists an open interval $I_0 \subset I$, a measurable set $A \subset J$ of positive Lebesgue measure, and $c > 0$ such that
\begin{equation}
\label{equation:ua-positive}
u_a(\Phi(t,y)) \geq c > 0, \qquad \forall (t,y) \in I_0 \times A.
\end{equation}
In particular, for all $(t,y) \in I_0 \times A$, $S_{\Phi(t,y)}M \cap \Lambda$ has Lebesgue measure at least $c$.
\end{corollary}

On the patch $U \subset M$, one may choose a smooth section $e : U \to SU$. An \emph{angular sector} $\Omega \subset \Ss^1$ is an interval of the circle; it defines a subset $S\Omega := \{(x,R_\theta e(x)) ~:~ \theta \in \Omega\} \subset SU$, where $R_\theta$ is the fiberwise rotation in $SM$ by angle $\theta$. One can then apply Corollary \ref{corollary:ai} with $a \in C^\infty(SM)$ such that $a|_{SU}$ is supported in $S\Omega$. Provided \eqref{equation:ua-positive} holds, Corollary \ref{corollary:ai} shows that there are ``many'' flat complete geodesics in $U$ with a tangent direction in $\Omega$.

Unfortunately, Corollary \ref{corollary:ai} does not show that $\{K=0\}$ must have non-empty interior provided $\Lambda$ has positive Liouville measure. For that, one would need to establish that $u_0 \in H^{1+}(M)$; by the embedding $H^{1+}(M) \hookrightarrow C^0(M)$, this would imply that $u_0$ coincides almost everywhere with a continuous function, hence $\{K=0\}$ has nonempty interior. 

\subsection{Microlocal regularity outside of the co-vertical space}

\label{ssection:microlocal}

%
%
%
%
%

The strategy of proof for Theorem \ref{theorem} is to establish that the characteristic function $\mathbf{1}_\Lambda \in L^\infty(SM)$ of $\Lambda$ is more regular than \emph{a priori} expected. Studying the regularity of solutions to partial differential equations is the realm of microlocal analysis. It is therefore natural to expect that progress towards Conjecture \ref{conjecture:folklore} could be obtained by means of microlocal methods.

\subsubsection{Cotangent bundle geometry} The tangent space to the unit tangent bundle $SM$ decomposes as
\[
T(SM) = \R X \oplus \HH \oplus \V,
\]
where $X$ is the geodesic vector field (generator of the geodesic flow), $\V$ is the \emph{vertical} distribution, namely, the tangent space to the circle fibers of $SM$ and $\HH$ is the \emph{horizontal} distribution, namely, the orthogonal to $\V \oplus \R X$ with respect to the natural Sasaki metric on $SM$. Similarly, the cotangent bundle $T^*(SM)$ can be decomposed as
\begin{equation}
\label{equation:splitting}
T^*(SM) = \R \alpha \oplus \HH^* \oplus \V^*,
\end{equation}
where $\alpha$ is the Liouville contact $1$-form (such that $\iota_X \alpha = 1, \iota_X d\alpha = 0$) and the two other distributions are defined by $\HH^*(\R X \oplus \V) = 0$, $\V^*(\R X \oplus \HH) = 0$.


\subsubsection{Microlocal regularity} We refer to \S\ref{ssection:analysis} for the definition of the wavefront set of a distribution and of a pseudodifferential operator. We will show the following key result (below, $\mu$ denotes the Liouville probability measure on $SM$):

\begin{theorem}
\label{theorem:main}
Let $(M,g)$ be a closed connected oriented Riemannian surface with nonpositive sectional curvature. Let $u \in L^2(SM,\mu)$ be an invariant function under the geodesic flow. Then for all $\eta > 0$, for all $A \in \Psi^0(SM)$ such that $\WF(A) \cap \V^* = \emptyset$, $A u \in H^{1/8-\eta}(SM)$. In particular, $u$ is microlocally $H^{1/8-}$ near $\HH^*$.
\end{theorem}

The proof of Theorem \ref{theorem:main} is based on a new \emph{radial estimate for parabolic flows}, see Proposition \ref{proposition:estimate} below. Theorem \ref{theorem:main} then implies Theorem \ref{theorem} by combining an averaging result on kinetic equations due to Gérard-Golse \cite{Gerard-Golse-92} with standard geometric estimates involving the Pestov identity. See \S\ref{section:conclusion}.

%
%
%


\subsection{Idea of proof}

The proof of Theorem \ref{theorem:main} is based on \emph{propagation of singularities}, which is a standard tool in the microlocal study of \emph{real principal type} operators such as vector fields. Roughly speaking, it asserts that if $u \in \mc{D}'(SM)$ solves $Pu = 0$ (for some operator $P$ satisfying certain natural assumptions), then the singularities of $u$ (its wavefront set) are invariant along the bicharacteristics of $P$, i.e., along the integral curves of the Hamiltonian vector field on $T^*(SM)$ induced by the principal symbol of $P$.

If $u \in L^2(SM)$ satisfies $X u = 0$, then standard elliptic theory guarantees that $\WF(u) \subset \Sigma := \HH^* \oplus \V^*$. Propagation of singularities implies that the wavefront set $\WF(u) \subset T^*(SM)$ of $u$ is invariant by the symplectic lift of the geodesic flow
\[
\Phi^{T^*(SM)}_t(v,\xi) := (\varphi_t(v), d\varphi_t^{-\top}(v)\xi),
\]
where ${}^{-\top}$ stands for the inverse transpose.

One approach to proving the ergodicity conjecture involves showing that $\WF(u) = \emptyset$, that is, $u \in C^\infty(SM)$. The transitivity of the flow would allow to conclude that $u$ is constant. The primary challenge in this method lies in establishing the existence of a direction $\xi_0 \in \Sigma$ such that $\xi_0 \notin \WF(u)$. This is what occurs for Anosov flows:  one first shows, using hyperbolic dynamics, that $\WF(u) \cap E_s^* = \emptyset$ (where $E_s^* \subset T^*(SM)$ denotes the (co)stable bundle) and similarly that $\WF(u) \cap E_u^*=\emptyset$ by simply reversing the dynamics, i.e., changing $X$ to $-X$. Finally, standard propagation of singularities shows that $\WF(u) = \emptyset$, implying that $u$ is smooth, hence constant. In nonpositive curvature, we aim to establish a similar outcome, specifically $\WF_{1/8-\eta}(u) \cap \HH^* = \emptyset$ for all $\eta > 0$, where $\WF_s$ denotes the wavefront set of a distribution in the $H^s$-regularity. However, this alone will not suffice to directly prove ergodicity of the geodesic flow. 

The pivotal aspect of the argument leading to $\WF_{1/8-\eta}(u) \cap \HH^* = \emptyset$ is the source-to-sink dynamics of the (transpose of the) differential of the geodesic flow on $\Sigma$ over the characteristic locus $\Lambda$. In the Anosov setting, the source corresponds to $E_s$, while the sink corresponds to $E_u$, and the attraction/repelling speed is exponential in time. In nonpositive curvature, a similar phenomenon still occurs; however, when restricted to the flat part, both the source and the sink collapse into the other horizontal direction and the attraction/repelling speed is polynomial in time (see Figure \ref{figure:cones} below).

\subsection{Organization of the paper} In \S\ref{section:preliminaries}, we introduce the geometric and dynamical background needed for the proofs. We review in \S\ref{section:analysis} semiclassical analysis and the FBI transform, and prove Theorem \ref{theorem:main} in \S\ref{section:proofs}. Finally, the consequences for the function $\mathbf{1}_\Lambda$ such as Theorem \ref{theorem} are derived in \S\ref{section:conclusion}. \\

\noindent \textbf{AI disclosure:} Generative AI was not used in the proof nor in the writing of the main results, namely, Theorem \ref{theorem}, Corollary \ref{corollary}, and Theorem \ref{theorem:main}. It was only used to suggest Corollary \ref{corollary:ai} after the author prompted the following question: ``What does it imply for a measurable set in $\R^2$ to support a nonzero, nonnegative function of regularity $H^{5/8-}$?''. \\

\noindent \textbf{Acknowedgement:} This manuscript was written during the summer 2022, and circulated in 2023 to a few colleagues. Over the last three years, I have occasionally tried to use Theorem \ref{theorem} to prove the Ergodicity Conjecture \ref{conjecture:folklore} but have not succeeded so far. With the flood of AI generated papers, I have decided to post this paper on the arXiv despite it not being fully conclusive so far. I believe it could be a first step towards proving the conjecture.

I am grateful to K. Burns for pointing out many missing references in an earlier draft, and to G. Knieper and F. Rodriguez Hertz for their comments. This project was supported by the European Research Council (ERC) under the European Union’s Horizon 2020 research and innovation programme (Grant agreement no. 101162990 — ADG).

\section{Geometric and dynamical preliminaries}

\label{section:preliminaries}

We refer to \cite[Chapter 1]{Paternain-99}, \cite{Knieper-02} or \cite[Chapters 13 and 15]{Lefeuvre-book} for an overview of standard facts on the geodesic flow.

\subsection{Geometry on the unit tangent bundle} Let $(M,g)$ be a closed connected oriented nonpositively-curved Riemannian surface of genus $\textsl{g} \geq 2$. The unit tangent bundle $\pi : SM \to M$ is a circle bundle over $M$ defined as
\[
SM := \left\{ (x,v) \in TM ~:~ |v|_g = 1\right\}.
\]
In what follows, a generic point in $SM$ is denoted by $v$. We will also use the compact notation $\M := SM$.

\subsubsection{Frames and coframes} Let $X \in C^\infty(\M,T\M)$ be the generator of the geodesic flow $(\varphi_t)_{t \in \R}$. Let $V \in C^\infty(\M,T\M)$ be the generator of the $\mathrm{SO}(2)$-action in the fibers, and set $H := [V,X]$. The frame $\left\{X,H,V\right\}$ is an orthonormal basis with respect to the \emph{Sasaki metric}. Define the \emph{vertical} distribution $\V := \R V$ and the \emph{horizontal} distribution\footnote{Usually, the horizontal distribution is $\R X \oplus \R H$. We adopt a slightly different terminology.} $\HH := \R H$.

The following commutation formulas hold:
\begin{equation}
\label{equation:commutateur}
[X,H] = K V, \qquad [V,X]= H, \qquad [H,V] = X,
\end{equation}
where $K$ is the sectional curvature. We also introduce the \emph{lowering} and \emph{raising} operators:
\begin{equation}
\label{equation:etapm}
\eta_\pm := \dfrac{1}{2}\left(X \mp iH \right).
\end{equation}
They satisfy the following commutation formulas:
\begin{equation}
\label{equation:commutateur2}
[-iV,\eta_+] = \eta_+, \qquad [-iV,\eta_-]=\eta_-, \qquad [\eta_+,\eta_-] = \dfrac{iK V}{2}.
\end{equation}

It will  be convenient to introduce the dual distributions to $\V,\HH$. Let $\alpha$ be the Liouville contact $1$-form such that $\alpha(X) = 1$, $\alpha(\HH \oplus \V)=0$. Define the $1$-forms $\beta,\gamma$ on $\M$ by
\[
\beta(H)=1, \beta(X)=\beta(V)=0, \qquad \gamma(V)=1,\psi(X)=\psi(H)=0.
\]
The following relations follow from \eqref{equation:commutateur}:
\begin{equation}
\label{equation:commutateur2}
d\alpha = \psi \wedge \beta, \qquad d\beta=\alpha\wedge\psi, \qquad d\psi =K  \beta \wedge \alpha.
\end{equation}
Define the \emph{co-horizontal} direction by $\HH^*=\R \beta$ and the \emph{co-vertical} direction by $\V^* = \R \psi$. Finally, we set $E_0^* := \R \alpha$. Any covector $\xi \in T^*\M$ can be decomposed as $\xi = \xi_\alpha \alpha + \xi_\beta \beta + \xi_\psi \psi$. 

\subsubsection{Normal bundle} Define the \emph{normal vector bundle} $\mc{N} \to \M$ as the trivial line bundle given at the point $v \in \M$ by $\R v^\perp$, where $v^\perp = e^{\pi/2 V}v$ is the vector $v$ rotated by an angle $+\pi/2$. Observe that there is a natural identification
\begin{equation}
\label{equation:identification}
\HH \oplus \V \simeq \mc{N}^{\oplus 2}, \qquad w_{H}H + w_{V} V \mapsto (w_{H}v^\perp,w_{V}v^\perp).
\end{equation}
This will be used in order to describe Jacobi vector fields and the differential of the geodesic flow.

\subsection{Dynamics on the unit tangent bundle} We now describe the dynamics of the geodesic flow on $\M$ and of the lifted flow to the tangent and cotangent bundles $T\M$ and $T^*\M$.

\subsubsection{Lifted dynamics}

Nonpositive curvature implies the existence of two flow-invariant continuous stable and unstable vector bundles $E_s, E_u \subset T\M$ defined as follows: there exists a constant $C > 0$ such that for all $v \in \M$:
\[
 w \in E_s(v) \iff  w \in \HH \oplus \V, \quad |d\varphi_t(v) w| \leq C |w|, \forall t \geq 0,
\]
and the same holds for $E_u$ by reversing time. We refer to \cite[Section 10.3]{Barreira-Pesin-13} for further details. These bundles integrate to smooth stable and unstable manifolds $W^{s,u}$ which are only transversally continuous.

An important feature of nonpositive curvature is that the distributions $E_s, E_u$ never intersect the vertical distribution:
\begin{equation}
\label{equation:vertical}
E_s \cap \V = E_u \cap \V = \left\{0\right\},
\end{equation}
see \cite[Proposition 10.3, item (2)]{Barreira-Pesin-13}. However, unlike in negative curvature, the bundles $E_s, E_u$ may not be everywhere transverse. Actually, it is a result of Eberlein \cite[Theorem 3.2]{Eberlein-73} that the geodesic flow is Anosov if and only if $E_s$ and $E_u$ are transverse everywhere. In this case, the geodesic flow is known to be ergodic. Over $\Lambda$, the bundles $E_s$ and $E_u$ are not transverse and they coincide with the horizontal distribution $\HH$.

The geodesic flow $(\varphi_t)_{t \in \R}$ lifts to the tangent bundle $T\M$ by taking $\Phi^{T\M}_t(v,w) := (\varphi_t(v), d\varphi_t(v)w)$. The differential of the geodesic flow can be computed using the Jacobi equation. More precisely, the following holds:
\begin{itemize}
\item for all $v \in \M, t \in \R$, $d\varphi_t(v)X(v) = X(\varphi_t(v))$;
\item for any $w \in \HH(v) \oplus \V(v)$, $d\varphi_t(v)w \in \HH(\varphi_t(v)) \oplus \V(\varphi_t(v))$; we can write $d\varphi_t(v)w = w_H(t) H + w_V(t) V$ where $w_H(t) = j(t), w_V(t) = \dot{j}(t)$, and $j$ solves the Jacobi equation
\begin{equation}
\label{equation:jacobi}
\ddot{j}(t) + K(t)j(t) = 0,
\end{equation}
where $K(t)$ is the sectional curvature at the point $\pi(\varphi_t(v)) \in M$. Note that, by the identification \eqref{equation:identification}, $d\varphi_t(v)w$ corresponds to $(j(t),\dot{j}(t))$.
\end{itemize}
In particular, as the curvature vanishes for all time over $\Lambda$, the Jacobi equation \eqref{equation:jacobi} allows to compute the expression for the lifted flow $(\Phi^{T\M}_t)_{t \in \R}$ over $\Lambda$. Decomposing $w = w_X X+w_H H+ w_V V$, one obtains
\begin{equation}
\label{equation:tm}
\Phi^{T\M}_t(v,w) = (\varphi_t(v), w_X X + (w_H + tw_V) H + w_V V).
\end{equation}

The dual distributions to $E_s$ and $E_u$ are defined by
\[
E_s^*(E_s \oplus \R X) = 0, \quad E_u^*(E_u \oplus \R X) = 0.
\]
They satisfy a non-intersection property similar to \eqref{equation:vertical}:
\begin{equation}
\label{equation:covertical}
E_s^* \cap \HH^* = E_u^* \cap \HH^* = \left\{0\right\}.
\end{equation}
The natural lift of the geodesic flow to the cotangent bundle $T^*\M$ is given by $\Phi_t^{T^*\M}(v,\xi) = (\varphi_t(v),d\varphi_t(v)^{-\top}\xi)$, where ${}^{-\top}$ stands for the inverse transpose. Decomposing $\xi = \xi_\alpha \alpha + \xi_\beta \beta + \xi_\psi \psi$, the expression of the lifted flow over $\Lambda$ is given by
\begin{equation}
\label{equation:tm-star}
\Phi^{T^*\M}_t(v,\xi) = (\varphi_t(v), \xi_\alpha \alpha + \xi_\beta \beta + (\xi_\psi - t\xi_\beta)\psi).
\end{equation} 

The \emph{characteristic set} $\Sigma := \HH^* \oplus \V^*$ will play an important role in what follows. Indeed, since $X$ has principal symbol $\sigma_X(v,\xi) = i \langle \xi, X(v) \rangle$ vanishing on $\Sigma$, it is elliptic on the complement $T^*\M \setminus \Sigma$ and thus, by standard elliptic theory, any distribution $u \in \mc{D}'(SM)$ in $\ker X$ (namely, invariant by the geodesic flow) has its wavefront set $\WF(u) \subset \Sigma$. In particular, any element $u \in \ker_{L^2}(X)$ has its wavefront set in $\Sigma$ so Theorem \ref{theorem:main} boils down to a microlocal statement in an arbitrarily small conic neighborhood of $\Sigma$.

\subsubsection{Dynamics on the radial compactification of the characteristic set}
\label{sssection:compactification}
It will be convenient to have in mind the action of the lifted geodesic flow on the radial compactification of $\Sigma$, especially over $\Lambda$. (We could also consider the radial compactification $\overline{T^*\M}$ of $T^*\M$ but since the action of the lifted flow $(\Phi_t^{T^*\M})_{t \in \R}$ is trivial in the $E_0^*$-direction, this does not really matter.) For that, define the smooth function $r := \langle \xi \rangle^{-1}$ on $\Sigma$. Then $\overline{\Sigma} := \Sigma \sqcup \left\{r = 0\right\}$ is a smooth manifold with boundary $\partial_\infty \Sigma = \left\{r = 0\right\}$ (the circle at infinity) with the smooth structure induced by taking $1/r$ as a boundary defining function. Since the lifted geodesic flow on $T^*\M$ is $1$-homogeneous, it induces a smooth flow on the radial compactification $\overline{\Sigma}$. Let $\mathbf{X}$ be the generator of $(\Phi^{T^*\M}_t)_{t \in \R}$. Then $\mathbf{X} \in C^\infty(T^*\M, T(T^*\M))$ extends to a smooth vector field on $\overline{\Sigma}$ and it is tangent to the boundary $\partial_\infty \Sigma$.

\begin{center}
\begin{figure}[htbp!]
\includegraphics{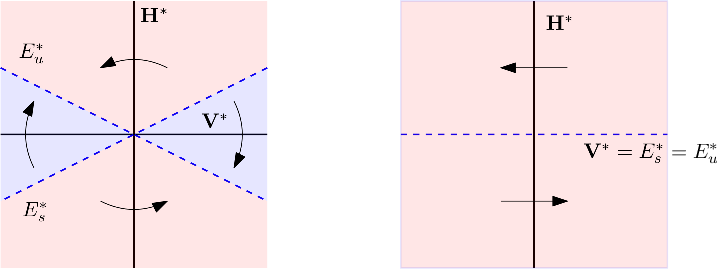}
\caption{The lifted dynamics on $\Sigma$. Left picture: dynamics near a point where $E_s^*$ and $E_u^*$ are transverse; the light blue region represent the cone $\mc{C}$. Right picture: dynamics over the set $\Lambda$ where $E_s^*=E_u^*=\V^*$; the cone $\mc{C}$ collapses here to a line.}
\label{figure:cones}
\end{figure}
\end{center} 

Since $\HH^* \oplus \V^* \to \M$ is trivial, the smooth manifold $\partial_\infty \Sigma$ is a trivial circle bundle over $\M$ which we may write
\begin{equation}
\label{equation:product}
\partial_\infty \Sigma = \M \times \Ss^{1}_\theta.
\end{equation}
Since $E_s^*$ and $E_u^*$ are lines in $\Sigma$, they define two cones in $\Sigma$ (see Figure \ref{figure:cones}). The following holds:

\begin{lemma}
\label{lemma:non-tangent}
There exists a closed flow-invariant cone $\mc{C} \subset \Sigma$ with edges given by $E_s^*$ and $E_u^*$ such that $\V^* \subset \mc{C}$ and $\HH^* \cap \mc{C} = \emptyset$. Furthermore, on $\partial_\infty \Sigma \setminus \mc{C}$, the $\partial_\theta$-component of $\mathbf{X}$ in the coordinates \eqref{equation:product} does not vanish.
\end{lemma}

Lemma \ref{lemma:non-tangent} describes the dynamics on $\Sigma$, see Figure \ref{figure:cones} for an illustration. It simply follows from the classical description of the geodesic flow in nonpositive curvature, see \cite{Knieper-02,Barreira-Pesin-13}. Note that, in particular, on $\partial_\infty \Sigma \cap \HH^*$, the $\partial_\theta$-component of $\mathbf{X}$ in \eqref{equation:product} does not vanish and this is independent of the sign of the sectional curvature of $(M,g)$: this is known as the \emph{twist property} of the geodesic flow.

Finally, it will be convient to have in mind an explicit computation for $\mathbf{X}$ on $\partial_\infty \Sigma \cap \pi^{-1}(\Lambda)$ near $\V^*$. Since $\V^*$ is a \emph{trapped set} for the lifted dynamics (that is, for any $(v,\xi) \in T^*\M$ with $\xi \in \V^*$, the $\xi$-component of $\Phi_t^{T^*\M}(v,\xi)$ has constant norm), one expects the $\partial_\theta$-component of $\mathbf{X}$ to vanish at $\V^*$. Indeed, introduce $\theta \in (-\pi/2,\pi/2)$ such that $\tan \theta = \xi_\beta/\xi_\psi$ for $\xi_\psi \gg 1$. Note that $(v,\theta)$ (for $v \in \M$ and $\theta  \in (-\pi/2,\pi/2)$) are well-defined coordinates on $\partial_\infty \Sigma$ near $\V^*_+ := \left\{ \lambda \psi ~:~ \lambda > 0\right\}$.

Then, $\Phi_t(v,\theta) = (\varphi_t(v),\theta_t)$ with
\[
\begin{split}
& \theta_t = \tan^{-1}((\xi_\beta)_t,(\xi_\psi)_t) = \tan^{-1}(\xi_\beta/(\xi_\psi-t\xi_\beta)), \\
&  \partial_t \theta_t|_{t=0} =  \dfrac{(\xi_\beta/\xi_\psi)^2}{1+  (\xi_\beta/\xi_\psi)^2} = \dfrac{\tan^2\theta}{1+\tan^2\theta},
\end{split}
\]
that is, one has over $\Lambda$:
\begin{equation}
\label{equation:order2}
\mathbf{X}(v,\theta) = X(v) + \dfrac{\tan^2\theta}{1+\tan^2\theta} \partial_\theta.
\end{equation}
We see that the vector field vanishes to second order on the covertical space $\V^*$ over $\Lambda$ and this will be crucial for the propagation of singularities argument below.

%

\subsection{Powers of the canonical bundle and Fourier decomposition}

\label{ssection:fourier}

\subsubsection{Fourier decomposition} Any smooth function $f \in C^\infty(\M)$ can be decomposed in Fourier series in the circle fibers of the unit tangent bundle $\M = SM$, that is, one has:
\begin{equation}
\label{equation:decomposition}
u = \sum_{k \in \Z} u_k,
\end{equation}
where $u_k \in C^\infty(\M)$ is globally defined and satisfies $V u_k = ik u_k$. The smooth functions $u_k$ are sections of the line bundle $\Omega_k \to M$ which consists of all smooth functions on $\M$ in $\ker(V-ik)$. Writing $\kappa \to M$ for the canonical line bundle over $M$ given by the complex structure of the surface (induced by the conformal class of the metric $g$), there is a natural identification $\pi_k^* :  \kappa^{\otimes k}\xrightarrow{\sim} \Omega_k$ for $k \in \Z$, given by $\pi_k^*(\omega)(x,v) := \omega_x(v^{\otimes k})$ for $\omega$ a section of $\kappa^{\otimes k}$ and $(x,v) \in SM$. The adjoint operator ${\pi_k}_* : L^2(M, \Omega_k) \to L^2(M,\kappa^{\otimes k})$ extends naturally to a map ${\pi_k}_* : L^2(SM) \to L^2(M,\kappa^{\otimes k})$ which is zero on all Fourier modes $\neq k$.

The splitting \eqref{equation:decomposition} also holds for $f \in L^2(\M)$ and the sum is orthogonal with respect to the $L^2$ scalar product on $\M$, that is
\[
\|u\|^2_{L^2(\M)} = \sum_{k \in \Z} \|u_k\|_{L^2(\M)}.
\]
More generally, any distribution $u \in \mc{D}'(SM)$ admits a decomposition \eqref{equation:decomposition} by setting for $\varphi \in C^\infty(\M)$: $\langle u_k, \varphi \rangle := \langle u, \varphi_k \rangle$.

\subsubsection{Raising and lowering operators}
The lowering/raising operators $\eta_\pm$ act on \eqref{equation:decomposition} as follows:
\[
\eta_\pm : C^\infty(M,\Omega_k) \to C^\infty(M,\Omega_{k \pm 1}).
\]
The line bundles $\kappa^{\otimes k} \to M$ are holomorphic line bundles equipped with $\partial$ and $\overline{\partial}$ operators. Under the identification $\Omega_k \simeq \kappa^{\otimes k}$, $\eta_+$ gets identified with $\partial$ for $k \leq 0$ and $\eta_-$ with $\overline{\partial}$ for $k \geq 0$, see \cite[Lemma 15.2.5]{Lefeuvre-book}.
Hence, these operators are elliptic and the dimension of their kernel can be explicitly computed, see \cite[Lemma 15.2.6]{Lefeuvre-book} for instance. For $k \geq 1$ (resp. $k \leq -1$), $\eta_+ : C^\infty(M,\Omega_k) \to C^\infty(M,\Omega_{k+1})$ (resp. $\eta_- : C^\infty(M,\Omega_{k}) \to C^\infty(M,\Omega_{k-1})$) is injective. Finally, the following \emph{Pestov identity} holds (see \cite[Lemma 15.2.4]{Lefeuvre-book}):
\begin{equation}
\label{equation:pestov}
\|\eta_+ u\|^2_{L^2(SM)} = \|\eta_-u\|^2_{L^2(SM)} - \dfrac{k}{2} \langle K u, u \rangle_{L^2(SM)}, \qquad \forall u \in C^\infty(M,\Omega_k), k \in \Z.
\end{equation}

\subsubsection{Geometric anisotropic Sobolev spaces} \label{ssection:spaces}

The \emph{geometric anisotropic Sobolev spaces} are defined as follows. For $s, t \in \R$, and $u \in C^\infty(SM)$, let
\[
\begin{split}
\|u\|^2_{H^s_xH^{t}_\theta(SM)} &:= \|(1-X^2-H^2)^{s/2}(1-V^2)^{t/2}u\|^2_{L^2(SM)} \\
&= \sum_{k \in \Z} (1+k^2)^{t}\|(1-X^2-H^2)u_k\|^2_{L^2(SM)}.
\end{split}
\]
The space $H^s_xH^{t}_\theta(SM)$ is defined as the completion of $C^\infty(SM)$ with respect to the previous norm. When $t=0$ (resp. $s=0$), we simply write $H^s_x(SM)$ (resp. $H^t_\theta(SM)$).

\section{Semiclassical analysis and radial estimates}

\label{section:analysis}

\subsection{Semiclassical analysis}

\label{ssection:analysis}

In the following two subsections, $\M$ is a closed connected ($n$-dimensional) manifold. We recall some standard elements of semiclassical analysis and refer the reader to \cite{Zworski-12} for a detailed account.

\subsubsection{Exotic symbols and pseudodifferential operators}

Let $\delta \in [0,1/2)$, $k \in \R$. We introduce the \emph{exotic} class of symbols $S^{k}_{\delta}(T^*\M)$ as the set of all smooth functions $a \in C^\infty(T^*\M \times (0,1]_h))$ satisfying the following estimates in every coordinate chart $U \subset \R^n$: for all $\alpha,\beta \in \N^n$, there exists a constant $C := C(\alpha,\beta,U) > 0$ such that
\begin{equation}
\label{equation:exotic}
\forall (x,\xi) \in T^*U, \forall h \in (0,1], \qquad |\partial_\xi^\alpha \partial_x^\beta a_h(x,\xi)| \leq C h^{-\delta(|\alpha|+|\beta|)} \langle \xi \rangle^{k-|\alpha|}.
\end{equation}
Although this will not be reflected in the notation later, symbols in $S^k_{\delta}(T^*\M)$ \emph{do} depend on the variable $h$, with the requirement that \eqref{equation:exotic} be satisfied. One can check that \eqref{equation:exotic} are equivariant by the action of $\mathrm{Diff}(\R^n)$ so it allows to define the classes $S^{k}_\delta(T^*\M)$ intrinsically on closed manifolds. When $\delta=0$, we will write $S^k(T^*\M)$.

We consider an arbitrary semiclassical quantization procedure $\Op_h$ on $\M$. In a local patch of coordinates $U \subset \R^n$, $\Op_h$ is given by
\begin{equation}
\label{equation:sc}
\Op_h(a)f (x) := \dfrac{1}{(2\pi h)^n} \int_{\R^n_\xi} \int_{\R^n_y} e^{\tfrac{i}{h} \xi \cdot (x-y)} a(x,\xi) f(y) |dy| |d\xi|,
\end{equation}
where $f \in C^\infty_{\mathrm{comp}}(U)$ and $a \in \mc{S}(\R^{2n})$ is Schwartz. To define $\Op_h$ on $\M$, take a cover of $\M$ by small open subsets and define the quantization using a partition of unity by \eqref{equation:sc} on each open subset, see \cite[Section 14.2]{Zworski-12} for further details.

More generally, one can quantize symbols in the class $S^{k}_\delta(T^*\M)$ for all $k \in \R$. We define:
\[
\Psi^{k}_{h,\delta}(\M) := \left\{ \Op_h(a) + \mc{O}_{\Psi^{-\infty}_h}(h^\infty) ~:~ a \in S^k_\delta(T^*\M) \right\}
\]
to be the set of all semiclassical exotic pseudodifferential operators of order $k \in \R$ (with parameter $\delta \in [0,1/2)$). As usual, $\Psi^\star_{h,\delta}(\M) := \cup_{k \in \R} \Psi^k_{h,\delta}(\M)$ forms a graded algebra. Moreover, there is a well-defined \emph{principal symbol} map
\[
\sigma : \Psi^k_{h,\delta}(T^*\M) \to S^k_\delta(T^*\M)/(h^{1-2\delta}S^{k-1}_\delta(T^*M)), 
\]
such that
\[
0 \to h^{1-\delta}\Psi^{k-1}_{h,\delta}(\M) \to \Psi^k_{h,\delta}(M) \overset{\sigma}{\to} S^k_\delta(T^*\M)/(h^{1-2\delta}S^{k-1}_\delta(T^*M)) \to 0
\]
is a long exact sequence. When $\delta = 0$, we will simply write $\Psi^k_h(\M)$ (standard $h$-pseudodifferential operators), and when $\delta=0, h =1$, we will write $\Psi^k(\M)$ ($h$-independent pseudodifferential operators).

We sum up the main properties of exotic operators:

\begin{enumerate}
\item \textbf{Composition of pseudodifferential operators:} given two symbols $a \in S^k_{\delta}(T^*\M)$, and $b \in S^{k'}_\delta(T^*\M)$, one has 
\begin{equation}
\label{equation:composition}
\Op_h(a) \Op_h(b) = \Op_h(ab) + \mc{O}_{\Psi^{k+k'-1}_{h,\delta}}(h^{1-2\delta}).
\end{equation}
\item \textbf{$L^2$-boundedness:} given $a \in S^0_\delta(T^*\M)$,
\[
\Op_h(a) : L^2(\M) \to L^2(\M)
\]
is bounded with norm given by
\begin{equation}
\label{equation:l2-bound}
\|\Op_h(a)\|_{L^2 \to L^2} \leq \|a\|_{L^\infty(T^*\M)} + \mc{O}(h^{1-2\delta}).
\end{equation}
\item \textbf{Egorov's Theorem:} given $a \in S^k_\delta(T^*\M)$ and $\psi \in \mathrm{Diff}(M)$, one has
\[
\psi^* \Op_h(a) (\psi^{-1})^* \in \Psi^k_{h,\delta}(\M),
\]
and
\begin{equation}
\label{equation:egorov}
\psi^* \Op_h(a) (\psi^{-1})^* = \Op_h(\Psi^* a) + \mc{O}_{\Psi^{k-1}_{h,\delta}(\M)}(h^{1-2\delta}),
\end{equation}
where $\Psi : T^*\M \to T^*\M$, defined by $\Psi(x,\xi) := (\psi(x),d\psi^{-\top}(x)\xi)$, is the symplectic lift of $\psi$ to $T^*\M$.
\end{enumerate}

By taking $h = 1, \delta = 0$ in the previous discussion, one retrieves the usual algebra of $h$-independent pseudodifferential operators on $\M$.

\subsubsection{Wavefront set}

\label{sssection:wavefront}

We now consider the case where $\delta=0,h=1$. The \emph{wavefront set} $\WF(A)$ (or the \emph{microsupport}) of an operator $A \in \Psi^{k}(\M)$ is the (closed) conic subset of $T^*\M \setminus \left\{0\right\}$ whose complement is characterized as follows: $(x_0,\xi_0) \notin \WF(A)$ if and only if for all $m \in \R$, for all $b \in S^m(T^*\M)$ supported in a small conic neighborhood of $(x_0,\xi_0)$, one has $A\Op(b) \in \Psi^{-\infty}(\M)$. In other words, the complement of the wavefront set of $A$ is the set of codirections where $A$ behaves as a smoothing operator.

The wavefront set $\WF(u)$ of a distribution $u \in \mc{D}'(SM)$ is the (closed) conic subset of $T^*\M \setminus \left\{0\right\}$ whose complement is characterized as follows: $(x_0,\xi_0) \notin \WF(u)$ if and only if there exists a small open conic neighborhood $V$ of $(x_0,\xi_0)$ such that for all $k \in \R$, for all $A \in \Psi^k(\M)$ with wavefront set contained in $V$, one has $A u \in C^\infty(\M)$. 
 
 The elliptic set $\mathrm{ell}(A) \subset T^*\M \setminus \left\{0\right\}$ of an operator $A \in \Psi^k(T^*\M)$ is defined as the (open) conic set of points $(x_0,\xi_0) \in T^*\M \setminus \left\{ 0\right\}$ such that there exists a constant $C > 0$ such that the following holds:
\begin{equation}
\label{equation:elliptic}
\begin{split}
|\xi| \geq C, d_{S^*\M}\left((x,\xi/|\xi|), (x_0,\xi_0/|\xi_0|) \right) < 1/C \qquad \implies |\sigma_A(x,\xi)| \geq \langle \xi\rangle^k/C.
\end{split}
\end{equation}
Here $d_{S^*\M}$ is any metric on the cosphere bundle $S^*\M := \R_+\backslash T^*\M$ and the $\R_+$-action is given by (radial) dilation in the fibers of $T^*\M$. An operator is said to be \emph{elliptic} if $\mathrm{ell}(A) = T^*\M \setminus \left\{0\right\}$. The \emph{characteristic set} $\Sigma(A)$ is then defined as
\[
\Sigma(A) := T^*\M \setminus \mathrm{ell}(A).
\]
The following lemma is classical and will play an important role in what follows:

\begin{lemma}[Elliptic regularity]
Let $A \in \Psi^k(\M)$ and $u \in \mc{D}'(\M)$ such that $Au \in C^\infty(\M)$. Then $\WF(u) \subset \Sigma(A)$.
\end{lemma}

In particular, taking $A := X \in \Psi^1(\M)$, a vector field, its principal symbol is $\sigma_X(x,\xi) = i\langle \xi, X(x) \rangle$ and thus its characteristic set is given by $\Sigma := \left\{\langle\xi,X(x)\rangle=0\right\}$. It implies that any distribution $u \in \mc{D}'(\M)$ satisfying $Xu = 0$ has wavefront set in $\Sigma$.

\subsubsection{Sobolev spaces}

\label{sssection:sobolev}

Semiclassical Sobolev spaces $H^s_h(\M)$ for $s \in \R$ are constructed as follows. Consider an arbitrary metric $g_{\M}$; let $\Delta_{g_{\M}} \geq 0$ be the associated Laplacian, and consider the completion of $C^\infty(\M)$ with respect to the norm
\begin{equation}
\label{equation:sobolev}
\|u\|_{H^s_h(\M)} := \|(\mathbbm{1}+h^2\Delta_{g_{\M}})^{s/2} u\|_{L^2(\M)},
\end{equation}
where $(\mathbbm{1}+h^2\Delta_{g_{\M}})^{s/2}$ is defined by the spectral Theorem (for unbounded selfadjoint operators) and $(\mathbbm{1}+h^2\Delta_{g_{\M}})^{s/2} \in \Psi^s_h(\M)$. For different values of $h,h'$ the Sobolev spaces $H^s_{h}(\M)$ and $H^{s}_{h'}(\M)$ are the same with equivalent norms. In particular, when $h=1$, we will simply write $H^s(\M)$. The following terminology will be used several times:

\begin{definition}
Let $Y \subset T^*M \setminus 0$ be a closed open subset. A distribution $u \in \mc{D}'(SM)$ is \emph{microlocally $H^s$ on $A$} if the following holds: for all $(x_0,\xi_0) \in Y$, there exists a small open conic neighborhood $V$ of $(x_0,\xi_0)$ such that for all $A \in \Psi^0(\M)$ with $\WF(A) \subset V$, one has $A u \in H^s(\M)$.
\end{definition}

It will be convenient to have at our disposal a Paley-Littlewood characterization of Sobolev spaces. Fix $R > 0$ and consider a nonnegative cutoff function $\chi_0 \in C^\infty_{\comp}(T^*\M)$ such that $\chi_0 \equiv 1$ for $\left\{|\xi| \leq R\right\}$ and $\chi_0 \equiv 0$ for $\left\{|\xi\ \geq 2R\right\}$. Define for $j \geq 1$:
\[
\chi_j(x,\xi) := \chi_0(x,2^{-j}\xi) - \chi_0(x,2^{-(j-1)}\xi).
\]
This function has support in $\left\{2^{j-1}R \leq 2^{j+1}R\right\}$ and is equal to $1$ on $|\xi|=2^j R$. Moreover, we have by construction $\mathbf{1} = \sum_j \chi_j$ on $T^*\M$. We then set for $u \in C^\infty(\M)$:
\begin{equation}
\label{equation:sob-besov}
\|u\|_{H^s_{\mathrm{aux}}}^2 := \sum_{j=0}^{+\infty} 4^{js} \|\Op(\chi_j)u\|_{L^2}^2
\end{equation}
It is a standard fact that \eqref{equation:sobolev} (with $h=1$) and \eqref{equation:sob-besov} define the same Sobolev spaces with equivalent norms. The choice of the constant $R > 0$ does not affect the space (it only changes the norm to another equivalent norm). Hence, in the following, we will simply use the notation $H^s(\M)$ equipped with one of these two norms. In the Paley-Littlewood decomposition \eqref{equation:sob-besov}, the exponent $2^{-j}$ should be interpreted as the semiclassical parameter $h$. In particular,
\begin{equation}
\label{equation:chi}
\chi_j(x,\xi) = \chi_1(x,2^{-(j-1)}\xi) =: \chi(x,2^{-j}\xi),
\end{equation}
where $\chi(x,\xi) = \chi_1(x,2\xi)$ (the function $\chi$ has support in $\left\{R/2 \leq |\xi| \leq 2R\right\}$ and is equal to $1$ on $\left\{|\xi|=R\right\}$). Thus
\begin{equation}
\label{equation:sob-besov2}
\|u\|_{H^s_{\mathrm{aux}}}^2 =  \|\Op(\chi_0)u\|_{L^2}^2 + \sum_{j=1}^{+\infty} 4^{js}\|\Op_{2^{-j}}(\chi)u\|^2_{L^2},
\end{equation}
where $\Op_{2^{-j}}$ stands for the semiclassical quantization with parameter $h=2^{-j}$.

We will need the following characterization of Sobolev regularity:

\begin{lemma}
\label{lemma:sobolev}
Let $u \in \mc{D}'(\M)$ be an \emph{$h$-independent} distribution on $\M$. Let $(x_0,\xi_0) \in T^*\M \setminus \left\{0\right\}$ and $c \in S^0(T^*M)$ be a nonnegative symbol such that $c = 1$ in a conic neighborhood of $(x_0,\xi_0)$ and $c=0$ outside the complement of a slightly larger conic neighborhood. Assume that $\|\Op_h(c) u\|_{L^2}^2 = \mc{O}(h^s)$. Then $u$ is microlocally $H^{s/2-}$ near $(x_0,\xi_0)$, that is, for all $\eps > 0$, for all $A \in \Psi^0(\M)$ with wavefront set near $(x_0,\xi_0)$, $A u \in H^{s/2-\eps}(\M)$.
\end{lemma}

The notation $u \in H^{s-}(\M)$ is a shorthand for $u \in H^{s-\eta}(\M)$ for all $\eta > 0$. 

\begin{proof}
We use the Sobolev norm \eqref{equation:sob-besov2} given in terms of the Paley-Littlewood decomposition. In the definition of the Sobolev norm \eqref{equation:sob-besov2}, pick $R > 0$ large enough so that $c \equiv 1$ in $\left\{|\xi| \geq R\right\}$ in a conic neighborhood of $(x_0,\xi_0)$. Take $A \in \Psi^0(\M)$ with wavefront set near $(x_0,\xi_0)$. Since $c \equiv 1$ on the intersection of the support of $\chi$ (defined by \eqref{equation:chi}) with the wavefront set of $A$, we can apply a parametrix construction\footnote{Note that $A$ is not $h$-pseudodifferential (it varies at wavelength $\sim 1$) but $\Op_h(\chi)$ is $h$-pseudodifferential. Nevertheless, it is not difficult to check that $\Op_h(\chi)A = \Op_h(\widetilde{\chi}) + \mc{O}_{\Psi^{-\infty}_h}(h^\infty)$ for some symbol $\widetilde{\chi}$ supported on the support of $\chi$ (see \cite[Lemma 2.1]{Bonthonneau-Lefeuvre-21} for instance). Thus the parametrix construction applies as we claim.}and find $B_h \in \Psi^0_h(\M)$ such that
\[
B_h \Op_h(c) = \Op_h(\chi) A + \mc{O}_{\Psi^{-\infty}_h}(h^\infty).
\]
Hence:
\[
\begin{split}
\|\Op_h(\chi)Au\|_{L^2}^2  = \|B_h \Op_h(c) u\|^2_{L^2} + \mc{O}(h^\infty) \leq C \|\Op_h(c)\|^2_{L^2}+ \mc{O}(h^\infty) \leq C h^s.
 \end{split}
\]
Hence, taking $h = 2^{-j}$ and applying \eqref{equation:sob-besov2}, we find that $A u \in H^{s/2-\eps}$ for any $\eps > 0$.
\end{proof}

%
%
%

\subsection{FBI transform}

We will need basic properties of the Fourier-Bros-Iagolnitzer (FBI) transform (or Bargmann transform). We refer to \cite{Wunsch-Zworski-01, Bonthonneau-Jezequel-20} for a more extensive treatment.

A generic point on $\M$ will be denoted by $x$, while a generic point on $T^*\M$ will be denoted by $\alpha = (\alpha_x,\alpha_\xi)$. The manifold $\M$ is endowed with an arbitrary metric $g$, hence a Riemannian volume density $|dx|$ while $T^*\M$ is endowed with a canonical volume form induced by its symplectic structure $|d\alpha_x d\alpha_\xi|$. Introduce the phase
\begin{equation}
\label{equation:phase}
\phi(\alpha,x) = -\alpha_\xi(\exp^{-1}_{\alpha_x}(x)) + \dfrac{i}{2} \langle \alpha_\xi \rangle d(\alpha_x,x)^2,
\end{equation}
(well-defined for $x$ close to $\alpha_x$), where $\exp$ is the exponential map given by ${g_{\M}}$.

We introduce the $h$-dependent class of symbols $S^{m,k}(T^*\M \times \M)$ satisfying the following bounds: in every coordinate chart $U \subset \R^n$, for all $\alpha,\beta, \gamma \in \N^n$, we require that there exists a constant $C := C(\alpha,\beta,\gamma,U) > 0$ such that:
\begin{equation}
\label{equation:symbols}
\begin{split}
\forall (\alpha_x,\alpha_\xi,x) \in T^*U \times U, \forall &h \in (0,1], \quad |\partial_{\alpha_\xi}^\alpha \partial_{\alpha_x}^\beta \partial_x^{\gamma} a_h(\alpha_x,\alpha_\xi,x)| \leq C h^{-m} \langle \alpha_\xi \rangle^{k-|\alpha|}.
\end{split}
\end{equation}
We will use the following characterization of the FBI transform:

\begin{theorem}
\label{theorem:fbi}
There exists an operator
\[
\mc{T}_h : L^2(\M) \to L^2(T^*\M) \cap C^\infty(T^*\M),
\]
such that for any quantization $\Op_h$ on $\M$, for all $p \in S^{k}_\delta(T^*\M)$, one has
\begin{equation}
\label{equation:link}
\Op_h(p)-\mc{T}_h^*p\mc{T}_h \in h^{1-2\delta} \Psi^{k-1}_{h,\delta}(\M),
\end{equation}
and:
\begin{equation}
\label{equation:isometry}
\mc{T}_h^*\mc{T}_h = \mathbf{1}.
\end{equation}
\end{theorem}

The operator $\mc{T}_h$ is given by:
\begin{equation}
\label{equation:fbi}
\mc{T}_h f(\alpha) := \int_M e^{\tfrac{i}{h}\phi(\alpha,x)} a(\alpha,x) f(x) |dx| + R_h,
\end{equation}
for some elliptic $a \in S^{3n/4,n/4}(T^*\M \times \M)$, where $\phi$ is the phase function \eqref{equation:phase} and $R_h : \mc{D}'(\M) \to \langle\alpha_\xi\rangle^{-\infty}L^2(T^*\M)$ has norm $\mc{O}(h^\infty)$. Theorem \ref{theorem:fbi} for $\delta = 0$ can be found in \cite{Wunsch-Zworski-01}. The extension to $\delta \in [0,1/2)$ follows from adapting their arguments.

Let us also draw some conclusions. The $L^2$-norm of a function can be computed by means of the $L^2$-norm of its FBI transform since $\mc{T}_h : L^2(\M) \to L^2(T^*\M)$ is an isometry, that is, $\|\mc{T}_h u\|_{L^2(T^*\M)} = \|u\|_{L^2(\M)}$. More generally, the FBI transform \eqref{equation:fbi} allows to quantize rough symbols and to define generalized pseudodifferential operators $P_h := \mc{T}_h^* p \mc{T}_h$ for $p \in L^\infty(T^*\M)$. By construction, these operators are bounded on $L^2(\M)$. This will be particularly useful in what follows in order to microlocalize near the singular set $\Lambda$ (which could be very wild). Moreover, if $p \in L^\infty(T^*\M)$ is such that $p \geq 0$, then $\langle P_h u, u \rangle_{L^2(\M)} \geq 0$ for all $u \in L^2(\M)$. (Combining this observation with \eqref{equation:link} yields a short proof of the sharp Gärding inequality \cite[Theorem 4.32]{Zworski-12}.)

\subsection{Radial estimates}

\label{section:proofs}

We now go back to the initial geometric setting, that is, $\M := SM$ is the unit tangent bundle of a nonpositively curved closed Riemannian surface and $X$ is the geodesic vector field. The main key analytic estimate is the following result, which we refer to as a \emph{parabolic radial estimate}:

\begin{proposition}[Parabolic radial estimate]
\label{proposition:estimate}
There exists a pseudodifferential operator $C_h := \Op_h(c)$, where $c=1$ in a small neighborhood of $\partial_\infty T^*\M \cap \HH^*$ and $c=0$ on the complement of a slightly larger neighborhood, such that the following holds: for all $s < 1/4$, there exists a constant $C := C(s) > 0$ such that for all $u \in L^2(\M)$ with $Xu \in L^2(\M)$ and $\mathrm{supp}(u) \subset \Lambda$, one has:
\begin{equation}
\label{equation:main-estimate}
\|C_h u\|^2_{L^2} \leq C \left(h^{-s}\|Xu\|_{L^2}^2 + h^s \|u\|_{L^2}^2 \right).
\end{equation}
\end{proposition}

%

The previous estimate does not use the genus $\geq 2$ assumption and holds on the flat torus as well. Before proving Proposition \ref{proposition:estimate}, let us show how it implies Theorem \ref{theorem:main}.

\begin{proof}[Proof of Theorem \ref{theorem:main}]
Let $u \in L^2(SM)$ such that $X u = 0$. Write $u := u_1 + u_2$ with $u_1 := \mathbf{1}_\Lambda \cdot u$ and $u_2 := \mathbf{1}_{SM\setminus \Lambda} \cdot u$. Note that $Xu_1 = Xu_2 = 0$ and $u_1, u_2 \in L^2(SM)$. By ergodicity of the geodesic flow on $SM \setminus \Lambda$ with respect to the Liouville measure (see \cite[Theorem 10.10]{Barreira-Pesin-13}), we find that $u_2 = c \cdot\mathbf{1}_{SM \setminus \Lambda} = c\cdot(1-\mathbf{1}_\Lambda)$ for some constant $c \in \C$. Hence, the claim reduces to showing that if $w \in L^2(SM)$, $\mathrm{supp}(w) \subset \Lambda$ and $X w = 0$, then $w$ is microlocally $H^{1/8-}$ outside of $\V^*$; it then suffices to apply this fact with $w=u_1$ and $w=\mathbf{1}_\Lambda$.

Suppose that $w$ is as above. Then \eqref{equation:main-estimate} yields $\|C_h w\|^2_{L^2} = \mc{O}(h^{1/4-\eps})$ (for any $\eps > 0$). In turn, by Lemma \ref{lemma:sobolev}, this implies that $w$ is microlocally $H^{1/8-}$ on a conic neighborhood of $\HH^*$. By standard propagation of singularities, it then implies that $u$ is microlocally $H^{1/8-}$ everywhere except on $\V^*$, that is, for all $A \in \Psi^0(\M)$ such that $\WF(A) \cap \V^* = \emptyset$, one has $Aw \in H^{1/8-}$. This proves the claim.
\end{proof}

%

\begin{proof}[Proof of Proposition \ref{proposition:estimate}]
We now prove Proposition \ref{proposition:estimate}. Define
\[
\Sigma^\pm := \left\{ \xi \in T^*\M ~:~  \pm \xi_\beta \geq 0 \right\}.
\]
Observe that it suffices to prove the estimate \eqref{equation:main-estimate} in Proposition \ref{proposition:estimate} with $c$ being microlocally equal to $1$ near $\partial_\infty T^*\M \cap \HH^* \cap \Sigma^+$: indeed, the same estimate also immediately applies for $\Sigma^-$ (by simply reversing $X$ to $-X$); it then suffices to sum the two inequalities. We will thus prove \eqref{equation:main-estimate} with $C_h$ microlocalized near $\partial_\infty T^*\M \cap \HH^* \cap \Sigma^+$.

Consider the following three functions:
\begin{enumerate}[label=(\roman*)]
\item $\chi_\alpha \in C^\infty_{\comp}(\R)$ is a nonnegative function such that $\supp(\chi_\alpha) \subset [-1,1]$, $\chi_\alpha \equiv 1$ on $[-1/2,1/2]$,
\item $\chi_\beta \in C^\infty(\R)$ is a nonnegative function such that $\chi_\beta \equiv 1$ on $[0,\infty)$, $\supp(\chi_\beta) \subset [-1,\infty)$,
\item $\chi_\psi \in C^\infty(\R)$ is a nonnegative function such that $\chi_\psi \equiv 1$ on $(-\infty,-2]$, $\supp(\chi_\psi) \subset (-\infty,-1]$.
\end{enumerate}
Define the symbol on $T^*\M$:
\begin{equation}
\label{equation:escape}
a(v,\xi) := \chi_\alpha(\xi_\alpha/\langle \xi_\beta\rangle) \chi_\beta(h^{-\delta} \xi_\beta/|\xi_\psi|) \chi_\psi(\eps^{-1} \xi_\psi/\langle \xi_\beta\rangle),
\end{equation}
where $\xi = \xi_\alpha \alpha + \xi_\beta \beta + \xi_\psi \psi$ and $\eps > 0$ will be chosen small enough (independently of $h$). The function $\chi_\beta(h^{-\delta} \xi_\beta/|\xi_\psi|)$ is irregular near $\xi = 0$ but $\chi_\psi(\xi_\psi/\langle \xi_\beta\rangle)$ vanishes on a $\mc{O}(1)$-neighborhood of $\xi=0$. Observe that $a$ belongs to an exotic symbol class as defined in \eqref{equation:exotic}. More precisely, the following holds:

\begin{lemma}
The symbol $a$ belongs to $S^{0}_{\delta}(T^*M)$.
\end{lemma} 

The proof is a straightforward computation; it is omitted. The symbol $a$ is depicted in Figure \ref{figure}. We call $a$ an \emph{escape function} over $\Lambda$. This is justified by the following lemma:

\begin{lemma}
\label{lemma:escape}
There exists $\eps > 0$ small enough such that the escape function $a$ in \eqref{equation:escape} satisfies the following properties:
\begin{enumerate}[label=\emph{(\roman*)}]
\item On $\Sigma^+ \cap \pi^{-1}(\Lambda)$, one has:
\begin{equation}
\label{equation:escape-prop1}
\forall v \in \Lambda, \forall \xi \in T^*_v\M \cap \Sigma^+, \forall t \geq 0, \qquad a(\Phi_{t}(v,\xi)) - a(v,\xi) \geq 0.
\end{equation}
\item On $\Sigma^- \cap \pi^{-1}(\Lambda)$, one has:
\begin{equation}
\label{equation:escape-prop2}
\forall v \in \Lambda, \forall \xi \in T^*_v\M \cap \Sigma^-, \forall t \geq 0, \qquad a(\Phi_{t}(v,\xi)) - a(v,\xi) \geq \mc{O}_t(h^{\delta}).
\end{equation}
\item There exists a conic neighborhood $\mc{A}$ of $\HH^* \cap \Sigma^+$ over $\M$ such that for $t_0 \gg 0$ large enough $\Phi_{t_0}^*a-a = 1$ on $\mc{A}$.
\end{enumerate}
\end{lemma}
 
 \begin{center}
\begin{figure}[htbp!]
\includegraphics[scale=0.8]{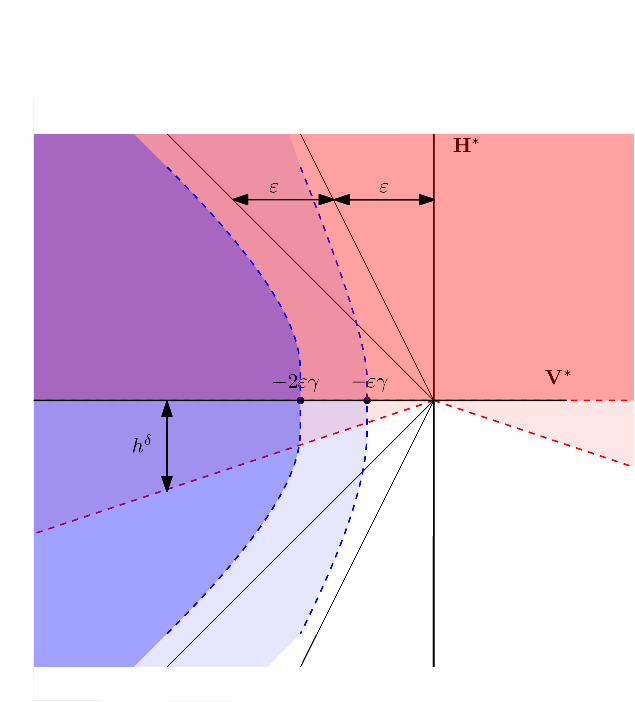}
\caption{The escape function $a$ on the characteristic set $\Sigma$. The red region represents the function $f := (\xi_\beta,\xi_\psi) \mapsto \chi_\beta(h^{-\delta} \xi_\beta/|\xi_\psi|)$: in light red, its support; in dark red, the set $\left\{f=1\right\}$. The blue region represents the function $g := (\xi_\beta,\xi_\psi) \mapsto \chi_\psi(\eps^{-1} \xi_\psi/\langle \xi_\beta\rangle)$: in light blue, its support; in dark blue, the set $\left\{g=1\right\}$. The support of $a$ is therefore represented as the intersection of the light red and light blue regions. The set $\left\{a=1\right\}$ is represented as the intersection of the dark blue and dark red regions. The black arrows measure the projective distance between the lines.}
\label{figure}
\end{figure}
\end{center} 

\begin{proof}
Recall from \eqref{equation:tm-star} that the lifted dynamics is given over $\Lambda$ by 
\[
\Phi_t(v,\xi) = (\varphi_t v, \xi_\alpha \alpha(\varphi_t v) + \xi_\beta \beta(\varphi_t v) + (\xi_\psi- t \xi_\beta) \psi(\varphi_t v))
\]
In $\Sigma^+ \cap \pi^{-1}(\Lambda)$, the expression for $a$ is given by
\[
a(v,\xi) := \chi_\alpha(\xi_\alpha/\langle \xi_\beta\rangle) \chi_\psi(h^{-\delta}\xi_\psi/\langle \xi_\beta\rangle)
\]
and it is thus immediate to verify that $a(\Phi_{t}(v,\xi)) - a(v,\xi) \geq 0$. This proves (i). Item (iii) is also immediate if we take $\eps > 0$ small enough due to the projective properties of the lifted flow near $\HH^* \cap \partial_\infty\Sigma^+$, see Lemma \ref{lemma:non-tangent} and \eqref{equation:covertical}.

The only non-trivial item to check is (ii). Recall that $\mathbf{X} \in C^\infty(T^*\M, T(T^*\M))$ denotes the generator of $(\Phi_t^{T^*\M})_{t \in \R}$. Then, on $\pi^{-1}(\Lambda)$, we have:
\begin{equation}
\label{equation:ha}
\begin{split}
\mathbf{X}a &= \chi_\alpha(\xi_\alpha/\langle \xi_\beta\rangle) \left( h^{-\delta} \left(\xi_\beta/\xi_\psi\right)^2 \chi_\beta'(h^{-\delta} \xi_\beta/|\xi_\psi|) \chi_\psi(\eps^{-1}\xi_\psi/\langle \xi_\beta \rangle) \right. \\
& \hspace{4cm} \left. - \eps^{-1} (\xi_\beta/\langle \xi_\beta \rangle) \chi_\beta(h^{-\delta} \xi_\beta/|\xi_\psi|) \chi_\psi'(\eps^{-1}\xi_\psi/\langle \xi_\beta \rangle) \right).
\end{split}
\end{equation}
Observe that on the support of $ \chi_\beta'(h^{-\delta} \xi_\beta/|\xi_\psi|)$ intersect with $\Sigma^-$, $|\xi_\beta/\xi_\psi| \leq h^\delta$ and thus the first term in \eqref{equation:ha} is bounded by $\mc{O}(h^\delta)$. As to the second term, we observe that on the support of $\chi_\beta(h^{-\delta} \xi_\beta/|\xi_\psi|) \chi_\psi'(\eps^{-1}\xi_\psi/\langle \xi_\beta \rangle)$, one has $|\xi_\beta| \leq Ch^\delta$. As a consequence, the second term is also bound by $\mc{O}(h^\delta)$. Integrating in time between $0$ and $t$, we obtain (ii). This concludes the proof.
\end{proof}

\begin{remark}
It should be observed that the important point in the proof is that the $\partial_\theta$ component of the flow on $\partial_\infty \Sigma$ vanishes to order $\geq 2$ near $\V^*$, as computed in \eqref{equation:order2}. This guarantees that item (ii) of Lemma \ref{lemma:escape} holds.
\end{remark}

Set $A_h :=  \mc{T}_h^* a \mc{T}_h$ and $B_{h} := e^{+t_0X}A_he^{-t_0X}$, where $t_0 > 0$ is given by Lemma \ref{lemma:escape}. We have:
\begin{equation}
\label{equation:ipp}
\begin{split}
\langle A_he^{-t_0X}u,& e^{-t_0 X} u\rangle_{L^2} -  \langle A_h u,u\rangle_{L^2} \\
& =  - \int_0^t \langle A_h e^{-sX}Xu, e^{-sX}u\rangle_{L^2} + \langle A_h e^{-sX}u, e^{-sX}Xu\rangle_{L^2} |ds|.
\end{split}
\end{equation}
Using Cauchy-Schwarz, the right-hand side is bounded (in absolute-value) by
\begin{equation}
\label{equation:rhs}
|\mathrm{rhs.}| \leq C \|Xu\|_{L^2}\|u\|_{L^2},
\end{equation}
where $C >0$ is some uniform constant depending on $A_h$ only.
Using the fact that the geodesic flow is volume-preserving, the left-hand side of \eqref{equation:ipp} can be written as
\[
\langle A_he^{-t_0X}u, e^{-t_0 X} u\rangle_{L^2} -  \langle A_h u,u\rangle_{L^2} = \langle (e^{t_0 X} A_h e^{-t_0 X}-A_h)u,u\rangle_{L^2},
\]
and by Egorov's Theorem, we have
\[
e^{t_0 X} A_h e^{-t_0 X}-A_h = \mc{T}_h^* (\Phi_{t_0}^*a-a) \mc{T}_h + \mc{O}_{\Psi^{-1}_{h,\delta}(\M)}(h^{1-2\delta}).
\]
As a consequence, the left-hand side of \eqref{equation:ipp} is equal to:
\begin{equation}
\label{equation:lhs}
\mathrm{lhs.} = \langle \mc{T}_h^*(\Phi_{t_0}^*a-a)\mc{T}_h u, u\rangle_{L^2} + \mc{O}(h^{1-2\delta}\|u\|_{H^{-1}_h}\|u\|_{L^2}).
\end{equation}

By the third item of Lemma \ref{lemma:escape}, there exists a symbol $c \in S^0(T^*\M)$ supported for $|\xi| \gg 1$ large enough in a conic neighborhood of $\HH^* \cap \Sigma^+$, such that $c \equiv 1$ on a (smaller) conic neighborhood of $\HH^* \cap \Sigma^+$ and $(\Phi_{t_0}^*a -a) - c \geq 0$ on the support of $c$. Hence, setting $C_h := \mc{T}_h^*c\mc{T}_h \in \Psi^0_h(\M)$, and combining both \eqref{equation:rhs} and \eqref{equation:lhs} in \eqref{equation:ipp}, we have:
\begin{equation}
\label{equation:middle}
\begin{split}
\langle \mc{T}_h^* ((\Phi_{t_0}^*a -a) - c) \mc{T}_h u, u \rangle  + \langle C_h u, u\rangle = \mc{O}(h^{1-2\delta}\|u\|_{H^{-1}_h}\|u\|_{L^2} + \|Xu\|_{L^2}\|u\|_{L^2}).
\end{split}
\end{equation}

We now aim to show that $\langle \mc{T}_h^* ((\Phi_{t_0}^*a -a) - c) \mc{T}_h u, u \rangle$ is bounded from below by a suitable power of $h$ as $h \to 0$. For that, let $\eps > 0$ and $\mathbf{1}_{\Lambda(\eps)}$ be the characteristic function of 
\[
\Lambda(\eps) := \{v \in SM ~:~ d(v,\Lambda) \leq \eps\}.
\]
Define $p \in L^\infty(T^*\M)$ by
\[
p(x,\xi) := \mathbf{1}_{\Lambda((h/\langle \xi \rangle)^{1/2-\eta})}(x),
\]
where $\eta > 0$ will be taken arbitrarily small. That is, at height $\xi$, we take the characteristic function of an $(h/\langle \xi \rangle)^{1/2-\eta}$-neighborhood of $\Lambda$. Observe that, since $(\Phi_{t_0}^*a -a) \in S^0(T^*\M)$ is $0$-homogeneous (for large $\xi$), varies at scale $\sim h^{-\delta}$ with respect to the $x$-variable (uniformly in $\xi$), and is $\geq - \mc{O}(h^\delta)$ on $\Lambda$ by the second item of Lemma \ref{lemma:escape}, there exists $C > 0$ such that for all $\eps > 0$ small enough, $(\Phi_{t_0}^*a -a)(x,\xi) \geq -C(h^\delta + h^{-\delta} \eps)$ for all $(x,\xi) \in T^*\M$ such that $x$ is in an $\eps$-neighborhood of $\Lambda$. As a consequence, we have:
\[
\begin{split}
&\langle \mc{T}_h^* ((\Phi_{t_0}^*a -a) - c) \mc{T}_h u, u \rangle \\
& = \langle \mc{T}_h^* \underbrace{((\Phi_{t_0}^*a -a) - c)p}_{\geq -C(h^\delta +h^{-\delta} (h/\langle \xi \rangle)^{1/2-\eta})}  \mc{T}_h u, u \rangle  + \langle \mc{T}_h^* (\Phi_{t_0}^*a -a) - c)(1-p) \mc{T}_h u, u \rangle \\
& \geq - C (h^\delta + h^{1/2-\eta-\delta})\|u\|^2_{L^2}  + \langle \mc{T}_h^* (\Phi_{t_0}^*a -a) - c)(1-p) \mc{T}_h u, u \rangle.
\end{split}
\]
Using that $\supp(u)\subset\Lambda$, we will now show that
\begin{equation}
\label{equation:bound}
\langle \mc{T}_h^* (\Phi_{t_0}^*a -a) - c)(1-p) \mc{T}_h u, u \rangle = \mc{O}(e^{-1/(Ch^{\eta'})}\|u\|^2_{L^2}),
\end{equation}
for some $C,\eta' > 0$. Observe that, by definition of $p$, the following holds: given $\alpha \in T^*\M$,
\[
1-p(\alpha) \neq 0 \implies \forall x \in \Lambda, \quad d(\alpha_x,x) \geq d(\alpha_x,\Lambda) > (h/\langle \alpha_\xi \rangle)^{1/2-\eta}.
\]
By definition \eqref{equation:phase} of the phase $\phi$, there exist constant $C_1,C_2 > 0$ large enough such that
\[
d(\alpha_x,x) \geq (h/\langle \alpha_\xi \rangle)^{1/2-\eta} \implies |e^{\tfrac{i}{h}\phi(\alpha,x)}| \leq C_2 e^{-C_1\tfrac{\langle \alpha_\xi \rangle}{h} \left( \tfrac{h}{\langle \alpha_\xi \rangle} \right)^{2(1/2-\eta)}}
\]
We thus obtain:
\[
\begin{split}
|(1-p(\alpha))\mc{T}_h u(\alpha)|  &= \left| \int_{\M} (1-p(\alpha)) e^{\tfrac{i}{h}\phi(\alpha,x)} a(\alpha,x) u(x) |dx| \right| \\
& \lesssim \left| \int_{\Lambda} |1-p(\alpha)|^2 e^{-C_1\tfrac{\langle \alpha_\xi \rangle}{h} \left( \tfrac{h}{\langle \alpha_\xi \rangle} \right)^{2(1/2-\eta)}} |a(\alpha,x)|^2 |dx| \right|^{1/2} \|u\|_{L^2} \\
& \lesssim  h^{-3n/4} \langle \alpha_\xi \rangle^{n/4} \exp\left(-C_1 \left( \dfrac{\langle \alpha_\xi \rangle}{h}\right)^{2\eta}\right)\|u\|_{L^2}.
\end{split}
\]
This yields
\[
\|(1-p)\mc{T}_h u\|_{L^2(T^*M)} = \mc{O}(e^{-1/(Ch^{\eta'})}\|u\|_{L^2}),
\]
for some other constants $C, \eta' > 0$. Finally, we obtain:
\[
\begin{split}
|\langle \mc{T}_h^* ((\Phi_{t_0}^*a -a) - c)(1-p) \mc{T}_h u, u \rangle| & \lesssim \|\mc{T}_h u\|_{L^2(T^*M)} \|(1-p) \mc{T}_h u\|_{L^2(T^*M)} \\
& \lesssim e^{-1/(Ch^{\eta'})} \|u\|^2_{L^2}.
\end{split}
\]
which proves \eqref{equation:bound}.

Going back to \eqref{equation:middle}, we obtain that
\[
\begin{split}
\langle C_h u, u \rangle_{L^2} \leq  C \big((h^{\delta} + h^{1-2\delta} + h^{1/2-\eta-\delta}  + e^{-1/(Ch^{\eta'})})\|u\|_{L^2}^2  + \|Xu\|_{L^2}\|u\|_{L^2}\big),
\end{split}
\]
and this holds for any $\eta > 0$. The optimal value is $\delta = 1/4$ which yields
\[
\langle C_h u, u \rangle_{L^2} =  \|C_h'u\|^2_{L^2} \leq C \|u\|_{L^2}\left(\|Xu\|_{L^2} + h^{1/4-\eta}\|u\|_{L^2}\right),
\]
for some other pseudodifferential operator $C_h' \in \Psi^0_h(\M)$, with same elliptic properties as $C_h$ near $\HH^* \cap \partial_\infty \Sigma^+$. Set $s := 1/4-\eta$. The estimate is finally obtained by bounding:
\[
\|u\|_{L^2}\|Xu\|_{L^2} \leq (h^{-s}\|Xu\|^2_{L^2} + h^s\|u\|^2_{L^2})/2.
\]
This concludes the proof.
\end{proof}

%
%

\section{Regularity of the singular locus} \label{section:conclusion}

\subsection{Regularity of Fourier modes}

We now further describe the regularity of Fourier modes of invariant functions.:

\begin{lemma}
\label{proposition:regularity-fourier-modes}
Let $u = \sum_{k \in \Z} u_k \in L^2(SM)$ such that $X u = 0$. Then
\[
u_k \in H^{5/8-}(M,\Omega_k), \qquad \forall k \in \Z.
\]
\end{lemma}

\begin{proof}
This is based on the following claim: let $u \in \mc{D}'(\M)$ such that $X u = 0$ and $u$ is microlocally $H^s$ near $\HH^*$; then, for all $k \in \Z$, $u_k \in H^{1/2+s}(M,\Omega_k)$. This is a direct consequence of \cite[Theorem 2.1]{Gerard-Golse-92}. The $y$-variable in \cite{Gerard-Golse-92} corresponds to the $\theta$-variable in our case; it thus suffices to check that $\mathbf{X}$ is transverse to $\left\{\langle \xi,V\rangle=0\right\}$ at each point of $\left\{\langle \xi,V\rangle=0\right\} \cap \left\{\langle \xi,X\rangle=0\right\} = \Sigma \cap \HH^*$. Transversality of $\mathbf{X}$ is a simple byproduct of Lemma \ref{lemma:non-tangent} (twist property of the geodesic flow). By Theorem \ref{theorem:main}, if $u \in L^2(SM)$ is flow-invariant, then $u$ is microlocally $H^{1/8-}$ near $\HH^*$ and the claim made at the beginning of this paragraph implies the result of the proposition.
\end{proof}

\subsection{Quantitative estimate} The previous result can be refined in a quantitative way:

\begin{proposition}
\label{proposition:quantitative}
Let $0 \leq \eps < 1/8$. Then there exists $C := C(\eps) > 0$ such that for all $u = \sum_{k \in \Z} u_k \in L^2(SM)$ such that $X u = 0$:
\[
\|u_k\|_{H^{1/2+\eps}_x(SM)}^2 \leq C\left(\langle k \rangle^{1/2+\eps}\|u\|_{L^2(SM)}^2 + \|u_0\|^2_{H^{1/2+\eps}_x(SM)}\right).
\]
\end{proposition}

The proof is based on an interpolation argument. The operator $\eta_+ : C^\infty(M,\Omega_k) \to C^\infty(M,\Omega_{k+1})$ being elliptic and injective for $k \geq 1$ (see \S\ref{ssection:fourier}), its adjoint $\eta_- = -\eta_+^* : C^\infty(M,\Omega_{k+1}) \to C^\infty(M,\Omega_k)$ is surjective. As a consequence, for all $v \in C^\infty(M,\Omega_{k})$, the equation $\eta_- w = v$ admits a unique solution $w \in C^\infty(M,\Omega_{k+1})$ such that $w$ is $L^2$-orthogonal to $\ker \eta_-|_{\Omega_{k+1}}$ (this also holds for Sobolev distributions). In particular, given $u \in C^\infty(M,\Omega_k)$, there exists a unique solution $B_k(u) \in C^\infty(M,\Omega_{k+2})$ such that:
\[
\eta_+ u + \eta_- B_k(u)=0, \qquad B(u) \perp \ker \eta_-|_{\Omega_{k+2}}.
\]
The operator $B_k$ is called the \emph{Beurling transform} (see \cite[Section 5]{Paternain-Salo-Uhlmann-15}). In addition, the following estimate holds:

\begin{lemma}
\label{lemma:beurling}
For all $k \geq 0$, $B_k : L^2(M,\Omega_k) \to L^2(M,\Omega_{k+2})$ is bounded and $\|B_k\|_{L^2\to L^2} \leq 1$.
\end{lemma}

The proof is standard; we include it for completeness.

\begin{proof}
For all $\varphi \in C^\infty(M,\Omega_{k+1})$, it holds:
\[
\begin{split}
|\langle B_k(u), \eta_+\varphi\rangle_{L^2}| & = |\langle\eta_-B_k(u),\varphi\rangle_{L^2}|=|\langle\eta_+u,\varphi\rangle_{L^2}| = |\langle u,\eta_-\varphi\rangle_{L^2}| \\
& \leq \|u\|_{L^2}\|\eta_-\varphi\|_{L^2} \leq \|u\|_{L^2} \|\eta_+\varphi\|_{L^2},
\end{split}
\]
where the last inequality follows from \eqref{equation:pestov} using that the sectional curvature is nonpositive and $k \geq 0$. As $\eta_+ : C^\infty(M,\Omega_{k+1}) \to C^\infty(M,\Omega_{k+2})$ is surjective, we can take the supremum over all $\varphi \in C^\infty(M,\Omega_{k+1})$ to retrieve the $L^2$-norm of $B_k(u)$. This yields $\|B_k(u)\|_{L^2} \leq \|u\|_{L^2}$.
\end{proof}

We introduce the operator
\[
T_k : C^\infty(M,\Omega_k) \times \ker \eta_-|_{\Omega_{k+2}} \to C^\infty(M,\Omega_{k+2}), \qquad T_k(u,h) := B_k(u) + h.
\]
The following estimate holds:

\begin{lemma}
\label{lemma:sioux}
For all $s \in [0,1]$, there exists a constant $C := C(s) > 0$ such that for all $u \in C^\infty(M,\Omega_k)$ and $h \in \ker \eta_-|_{\Omega_{k+2}}$:
\[
\|T_k(u,h)\|_{H^s_x(SM)}^2 \leq \|u\|^2_{H^s_x(SM)} + C\langle k \rangle^s(\|u\|^2_{L^2(SM)} + \|h\|^2_{L^2(SM)}).
\]
\end{lemma}

To avoid cumbersome notation in this proof, we suppress the index $\|\bullet\|_{L^2}$ for the $L^2$ norm.

\begin{proof}
This is established by interpolation. For $s=0$, one has, using Lemma \ref{lemma:beurling}:
\begin{equation}
\label{equation:l2}
\|T_k(u,h)\|^2 = \|B(u)\|^2 + \|h\|^2 \leq \|u\|^2 + \|h\|^2.
\end{equation}
For $s=1$, it holds, using \eqref{equation:pestov}:
\begin{equation}
\label{equation:h1}
\begin{split}
\|T_k(u,h)\|^2_{H^1_x(SM)} & = \|T_k(u,h)\|^2 + 2 (\|\eta_-T_k(u,h)\|^2 + \|\eta_+T_k(u,h)\|^2) \\
& = \|T_k(u,h)\|^2 + 2 (2\|\eta_-T_k(u,h)\|^2 - \tfrac{k}{2}\langle K T_k(u,h), T_k(u,h)\rangle) \\ 
& \leq 4\|\eta_+ u\|^2 + (Ck+1)\|T_k(u,h)\|^2 \\
& \leq 2\|\eta_- u\|^2 + 2\|\eta_+u\|^2 + (Ck+1)\|T_k(u,h)\|^2 \\
& \leq \langle ((Ck+1) - X^2 - H^2)u,u\rangle_{L^2} + (Ck+1)\|h\|^2,
\end{split}
\end{equation}
where the constant $C > 0$ might change from line to line.

We then claim that for all $s \in [0,1]$:
\begin{equation}
\label{equation:interpolation}
\|T_k(u,h)\|^2_{H^s_x(SM)} \leq \langle ((Ck+1) - X^2 - H^2)^su,u\rangle + (Ck+1)^s\|h\|^2
\end{equation}
Indeed, letting $A_s(u,h) := ((Ck+1) - X^2 - H^2)^{-s/2}u, (Ck+1)^{-s/2}h)$, we can rewrite \eqref{equation:h1} as:
\[
\|\langle\Delta \rangle^{1/2} T_k A_1(u,h)\|^2 \leq \|u\|^2 + \|h\|^2.
\]
The inequality \eqref{equation:l2} corresponds to $\|\langle\Delta \rangle^{0} T_k A_0(u,h)\|^2 \leq \|u\|^2 + \|h\|^2$. Applying Hadamard's three-lines theorem to the holomorphic function $s \mapsto \|\langle\Delta \rangle^{s} T_k A_s(u,h)\|^2$ which is bounded by $\|u\|^2+\|h\|^2$ on $\{\Re(s)=0\}$ and $\{\Re(s)=1\}$, we obtain that
\[
\|\langle\Delta \rangle^{s/2} T_k A_s(u,h)\|^2 \leq \|u\|^2 + \|h\|^2, \forall s \in [0,1].
\]
This is the claimed result \eqref{equation:interpolation}.

Finally, using that
\[
\langle(Ck+1) - X^2 - H^2)^su,u\rangle \leq \langle (1-X^2-H^2)^s u,u \rangle + (Ck+1)^s \|u\|^2,
\]
we obtain from \eqref{equation:interpolation} the desired inequality.
\end{proof}

Finally, we can prove Proposition \ref{proposition:quantitative}:

\begin{proof}[Proof of Proposition \ref{proposition:quantitative}]
Let $u = \sum_{k \in \Z} u_k \in L^2(SM)$ such that $Xu=0$. By Lemma \ref{proposition:regularity-fourier-modes}, $u_k \in H^{5/8-}(SM)$. We write $u_k = f_k + h_k$, where $f_k \in \ran \eta_+$ and $h_k \in \ker \eta_-$. Since $\eta_+u_k + \eta_- u_{k+2}$, we have $u_{k+2} = T(u_k, h_{k+2})$. Applying Lemma \ref{lemma:sioux}, we find for $0 \leq s < 5/8$:
\[
\begin{split}
\|u_{k+2}\|^2_{H^s_x(SM)} & \leq \|u_k\|^2_{H^s_x(SM)} + C\langle k \rangle^s(\|u_k\|^2_{L^2(SM)} + \|h_{k+2}\|^2_{L^2(SM)}) \\
&  \leq \|u_k\|^2_{H^s_x(SM)} + C\langle k \rangle^s(\|u_k\|^2_{L^2(SM)} + \|u_{k+2}\|^2_{L^2(SM)}) 
\end{split}
\] 
Iterating this inequality, we obtain the claimed result.
\end{proof}

\subsection{End of proof} We can now complete the proof of Theorem \ref{theorem}:

\begin{proof}[Proof of Theorem \ref{theorem}]
Let $u \in L^2(SM)$ such that $X u = 0$ and decompose $u = \sum_{k \in \Z} u_{k}$. By Lemma \ref{proposition:regularity-fourier-modes}, $u_{k} \in H^{1/2+\eps}_x(SM)$ for every $0 \leq \eps < 1/8$ and $k \in \Z$. It then follows immediately from Proposition \ref{proposition:quantitative} that for every $s > 3/4 + \eps/2$:
\[
\begin{split}
\|u\|^2_{H^{1/2+\eps}_x H^{-s}_\theta(SM)} &= \|(1-X^2-H^2)^{1/4+\eps/2}(1-V^2)^{-s/2}u \|^2_{L^2(SM)} \\
&  = \sum_{k \in \Z} (1+k^2)^{-s} \|u_{k}\|_{H^{1/2+\eps}_x(SM)}^2\\
& \leq C \sum_{k \in \Z} \langle k \rangle^{1/2+\eps} \langle k \rangle^{-2s} < +\infty.
\end{split}
\]
That is $\mathbf{1}_\Lambda \in H^{1/2+\eps}_x H^{-s}_\theta(SM)$.
\end{proof}

\subsection{Vanishing locus of the sectional curvature} Finally, we prove Corollary \ref{corollary:ai}.

\begin{proof}[Proof of Corollary \ref{corollary:ai}]
We consider the case $a \equiv 1$; the general case follows by similar arguments. Let $1/2 < s < 5/8$. Observe that there exists $C > 0$, independent of $u_0$, such that:
\[
\int_J \|u_0(\bullet,y)\|^2_{H^s(I)} |dy| \leq C \|u_0\|^2_{H^s(I \times J)}.
\]
Therefore, for almost every $y \in J$, $t \mapsto u_0(\Phi(t,y))$ belongs to $H^s(I)$. Since $s > 1/2$, $H^s(I) \hookrightarrow C^{s-1/2}(I)$ is a continuous embedding, and $u_0(\Phi(\bullet,y))$ admits a continuous representative $f_y \in C^{s-1/2}(I)$. Since
\[
\int_U u_0(x) |dx| > 0,
\]
there is a subset $J_0 \subset J$ of positive measure such that for all $y \in J_0$, $f_y \geq c_y > 0$ on a subinterval $I_y \subset I$.

We claim that this can be upgraded to the following statement: there exist $c > 0$ and a closed interval $I_0 \subset I$ such that for all $y \in J_0$, $f_y(t) \geq c$ for all $t \in I_0$. Indeed, for $p,q \in \Q$ and $n \in \Z_{> 0}$, consider
\[
A_{p,q,n} := \{y \in J_0 ~:~ f_y(t) \geq 1/n, \forall t \in [p,q]\}.
\]
By construction, $J_0 = \cup_{p,q,n} A_{p,q,n}$, so one of these sets has positive measure, which proves the claim with the corresponding $I_0 = [p,q]$ and $c=1/n$.

We now show that $u_0$ is upper semicontinuous. As $\Lambda \subset SM$ is closed, its characteristic function is upper semicontinuous. Therefore, given $x \in M$ and $x_n \to x$, we have by Fatou's lemma:
\[
\begin{split}
\limsup_{x_n \to x} u_0(x_n) & \leq \int_{S_{x_n}M} \limsup_{x_n \to x} \mathbf{1}_\Lambda(x_n,\theta) |d\theta|/2\pi \\
& \leq \int_{S_xM} \mathbf{1}_\Lambda(x,\theta) |d\theta|/2\pi = u_0(x).
\end{split}
\]

Now, fix $t \in I_0$ and $t_j \to t$ such that $u_0(\Phi(t_j,y)) = f_y(t_j)$. By upper semicontinuity of $u_0$, we then obtain that:
\[
u_0(\Phi(t,y)) \geq f_y(t) \geq c.
\]
This proves the claim.
\end{proof}

\bibliographystyle{alpha}
\bibliography{Biblio}

\end{document}